\documentclass[12pt]{amsart}

\usepackage{amssymb}

\usepackage{enumitem}

\usepackage{amsmath,amssymb,amsfonts,amsthm,color}
\usepackage{xcolor}
\usepackage{float}
\newtheorem{maintheorem}{Theorem}

\newcommand{\sing}{\textsf{Sing}}

\newcommand{\F}{\mathcal{F}}
\newcommand{\C}{\mathbb{C}}
\usepackage[dvips]{graphicx}
\usepackage{graphicx}
\usepackage{color}

\newcommand{\ze}{\mathbb{Z}}

\newcommand{\cl}[1]{\mathcal{#1}}

\newcommand{\cpt}[1]{\mathbb{C}P^{2}}

\newcommand{\D}{\mathcal{D}}
\newcommand{\B}{\mathcal{B}}
\newcommand{\Ftilde}{\tilde{\mathcal{F}}}
\newcommand{\G}{\mathcal{G}}

\newcommand{\sep}{\mbox{\rm Sep}}

\newcommand{\iso}{\mbox{Iso}}
\newcommand{\dic}{\mbox{Dic}}
\newcommand{\divv}{\mbox{\rm Div}}

\newcommand{\ord}{\textsf{ord}}

\newcommand{\val}{\text{Val}}

\makeatletter
\@namedef{subjclassname@2020}{%
  \textup{2020} Mathematics Subject Classification}
\makeatother

\usepackage[T1]{fontenc}

\newtheorem{theorem}{Theorem}[section]
\newtheorem{corollary}[theorem]{Corollary}

\newtheorem{proposition}[theorem]{Proposition}

\theoremstyle{definition}
\newtheorem{definition}[theorem]{Definition}
\newtheorem{remark}[theorem]{Remark}
\newtheorem{example}[theorem]{Example}

\numberwithin{equation}{section}

\begin{document}


\baselineskip=17pt


\title[On a Mattei-Salem theorem]{On a Mattei-Salem theorem}

\author[A. Fern\'andez-P\'erez]{Arturo Fern\'andez-P\'erez}
\address{Department of Mathematics\\ Federal University of Minas Gerais\\
Av. Pres. Ant\^onio Carlos, 6627 \\
CEP 31270-901, Pampulha, Belo Horizonte, Brazil}
\email{fernandez@ufmg.br}

\author[N. Saravia-Molina]{Nancy Saravia-Molina}
\address{Dpto. Ciencias - Secci\'{o}n Matem\'{a}ticas\\ Pontificia Universidad Cat\'{o}lica del Per\'{u}\\
Av. Universitaria 1801, San Miguel, Lima 32, Peru.}
\email{nsaraviam@pucp.edu.pe}

\date{}

\begin{abstract}
We investigate the relationship between the valuations of a germ of a singular foliation $\mathcal{F}$ on the complex plane and those of a balanced equation of separatrices for $\mathcal{F}$, extending a theorem by Mattei-Salem. Under certain conditions, we also derive inequalities involving the valuation, tangency excess, and degree of a holomorphic foliation $\mathcal{F}$ on the complex projective plane.
\end{abstract}
\dedicatory{In memory of Arkadiusz P\l oski}


\maketitle

\section{Introduction}
In \cite{mattei}, J.-F. Mattei and E. Salem formulated a theorem that characterizes germs of non-dicritical second type foliations (possibly formal) on $(\C^2,0)$ in terms of their algebraic multiplicity, valuation over components contained in the exceptional divisor of a reduction of singularities, and by an exact sequence of sheaves associated to these foliations. We recall that a foliation $\F$ is \textit{non-dicritical} if the number of local separatrices -- local irreducible invariant curves -- is finite, otherwise it is called \textit{dicritical}. A \textit{second type foliation} is a foliation $\F$ that admits -- at most -- \textit{non-tangent saddle-nodes}, meaning that no weak separatrix is contained in the exceptional divisor of a reduction of singularities of $\F$. In Section \ref{basic}, we review some of the standard facts on foliations, singularities, and their separatrices. 
\par In this paper, we study the \textit{valuation} $\nu_D(\F)$ of a germ of a singular foliation $\F$ at $p\in\C^2$ along a component $D$ contained in the exceptional divisor $\D$ of a minimal reduction of singularities of $\F$ (see Definition \ref{defi1}). Our aim is to generalize \cite[Theorem 3.1.9, item (4)]{mattei} to an arbitrary foliation (dicritical or not). It is worth pointing out that Y. Genzmer in \cite[Lemma 3.2]{genzmer} extended \cite[Theorem 3.1.9, item (4)]{mattei} to dicritical second type foliations (see also \cite[Proposition 3.7]{mol2}). Both, Mattei-Salem and Genzmer used this valuation to solve the \textit{realization problem} for holomorphic foliations, see for example \cite[Theorems 1.1 and 1.2]{genzmer}.
\par To state our main result, we introduce $\xi_D(\F)$ in (\ref{def_x}) and (\ref{def-xi}). This number is associated with the sum of the \textit{tangency excess} of $\F$ at \textit{infinitely near points} of $D$. Sections \ref{tangency} and \ref{valuation} provide a detailed exposition of these definitions. 
\par We can now formulate our main result:
\begin{maintheorem}\label{teo1}
Let $\F$ be a germ of a singular foliation at $p\in\C^2$ having $\hat{\Psi}$ as a balanced equation of separatrices. Let $\pi:(\tilde{X},\D)\to(\C^2,p)$ be a minimal process of reduction of singularities for $\F$. Then, for every component $D\subset\D$, we have
\begin{equation}
\nu_D(\hat{\Psi})=
\begin{cases}
\nu_D(\F)+1-\xi_D(\F) & \text{if $D$ is non-dicritical};
\medskip \\
\nu_D(\F) -\xi_D(\F) & \text{if $D$ is dicritical}.
\end{cases}
\end{equation}
\end{maintheorem}
\par Since $\F$ is of second type if and only if $\xi_p(\F)=0$ (see Definition \ref{defisecond}), Theorem \ref{teo1} generalizes \cite[Th\'eor\`em 3.1.9, item (4)]{mattei} and \cite[Lemma 3.2]{genzmer}, as $\xi_D(\F)=0$ by Definition \ref{def-2ndclass}. 
Finally, in Section \ref{projective}, using \cite[Theorem 1]{mendes} for projective foliations, we will derive inequalities involving $\nu_D(\F)$, $\xi_D(\F)$, and the degree of a holomorphic foliation $\F$ on $\mathbb{P}^2_{\C}$.

\section{Basic tools}\label{basic}

\par Let $\F$ be a germ of a singular foliation (possibly formal) at $p\in\C^2$. In local coordinates $(x,y)\in\C^ 2$ centered at $p$, the foliation is represented by a germ of a 1-form
\begin{equation}
\label{vectorfield}
\omega=P(x,y)dx+Q(x,y)dy,
\end{equation}
or by its dual vector field
\begin{equation}
\label{oneform}
v = -Q(x,y)\frac{\partial}{\partial{x}} + P(x,y)\frac{\partial}{\partial{y}},
\end{equation}
where  $P, Q   \in {\mathbb C}[[x,y]]$ are relatively prime.  The \textit{algebraic multiplicity} $\nu_p(\F)$ is the minimum of the order $\nu_p(P)$, $\nu_p(Q)$ at $p$ of the coefficients of a local generator of $\F$. 

\par Let $f(x,y)\in  \mathbb{C}[[x,y]]$. We say that $C: f(x,y)=0$  is {\em invariant} by $\F$ if $$\omega \wedge d f=(f\cdot h) dx \wedge dy,$$ for some  $h\in \mathbb{C}[[x,y]]$. If $C$ is irreducible, then we will say that $C$ is a {\em separatrix} of $\F$. The separatrix $C$ is analytical if $f$ is convergent.  We denote by $\sep_p(\F)$ the set of all separatrices of $\F$. 
\par We say that $p\in\C^2$ is a \textit{reduced} singularity for $\F$ if the linear part $\text{D}v(p)$ of the vector field $v$ in (\ref{oneform}) is non-zero and has eigenvalues $\lambda_1,\lambda_2\in\C$ fitting in one of the cases:
\begin{enumerate}
\item[(i)] $\lambda_1\lambda_2\neq 0$ and $\lambda_1\lambda_2\not\in\mathbb{Q}^{+}$ (\textit{non-degenerate});
\item[(ii)] $\lambda_1\neq 0$ and $\lambda_2\neq 0$ (\textit{saddle-node singularity}).
\end{enumerate}
In the case $(i)$, there is a system of coordinates  $(x,y)$ in which $\F$ is defined by the equation
\begin{equation}
\label{non-degenerate}
\omega=x(\lambda_1+a(x,y))dy-y(\lambda_2+b(x,y))dx,
\end{equation}
where $a(x,y),b(x,y)  \in {\mathbb C}[[x,y]]$ are non-units, so that  $\sep_p(\F)$ is formed by two
transversal analytic branches given by $\{x=0\}$ and $\{y=0\}$. In the case $(ii)$, up to a formal change of coordinates, the  saddle-node singularity is given by a 1-form of the type
\begin{equation}
\label{saddle-node-formal}
\omega = x^{k+1} dy-y(1 + \lambda x^{k})dx,
\end{equation}
where $\lambda \in \mathbb{C}$ and $k \in \mathbb{Z}^{+}$ are invariants after formal changes of coordinates (see \cite[Proposition 4.3]{martinetramis}).
The curve $\{x=0\}$   is an analytic separatrix, called {\em strong} separatrix, whereas $\{y=0\}$  corresponds to a possibly formal separatrix, called {\em weak} separatrix. The integer $k+1$ is called the  \textit{tangency index} of $\F$ with respect to the weak separatrix. 

 \par For a fixed minimal reduction process of singularities $\pi:(\tilde{X},\D)\to(\C^2,p)$ of $\F$ (it always exists, as established by Seidenberg \cite{seidenberg}), a component  $D \subset \D$ can be:
\begin{itemize}
\item {\em non-dicritical} if $D$ is $\tilde{\F}$-invariant. In this case, $D$ contains a finite number of simple singularities. Each non-corner singularity carries a separatrix   transversal to $D$, whose projection by $\pi$ is a curve in $\sep_{p}(\cl{F})$.
\item {\em dicritical} if $D$ is not $\tilde{\F}$-invariant. The definition
of reduction of singularities gives that $D$ may intersect only non-dicritical components and that $\tilde{\F}$ is everywhere transverse to $D$. The $\pi$-image of a local leaf of $\tilde{\F}$ at each non-corner point of $D$ belongs to $\sep_{p}(\cl{F})$.
\end{itemize}
\par Let $\sigma$ be the blow-up of the reduction process $\pi$ of $\F$ that generated the component $D\subset\D$. We will say that $\sigma$ is \textit{non-dicritical} (respectively \textit{dicritical}) if $D$ is non-dicritical (respectively dicritical).

\par Denote by   $\sep_{p}(D) \subset \sep_{p}(\cl{F})$ the set of separatrices whose strict transforms
 by $\pi$ intersect the
component $D \subset \D$. If $B \in \sep_{p}(D)$ with $D$ non-dicritical, $B$ is said to be \textit{isolated}. Otherwise, it is said to be a \textit{dicritical separatrix}.
This engenders the   decomposition $\sep_{p}(\cl{F}) = \iso_{p}(\cl{F}) \cup \dic_{p}(\cl{F})$, where notations are self-evident.
The set $\iso_{p}(\cl{F})$  is finite and  contains all   purely
formal separatrices. It   subdivides further    in two classes:
 \textit{weak} separatrices --- those arising from the weak separatrices of saddle-nodes --- and \textit{strong} separatrices --- corresponding to strong separatrices
of saddle-nodes and separatrices of non-degenerate singularities. On the other hand, if non-empty, $\dic_{p}(\cl{F})$ is an infinite set of analytic
separatrices. A foliation  $\cl{F}$ is said to be {\em  dicritical}
  when $\sep_{p}(\cl{F})$ is infinite, which is equivalent to saying that $\dic_{p}(\cl{F})$ is non-empty. Otherwise, $\cl{F}$ is called {\em non-dicritical}.

Along the text, we would rather adopt the language of \textit{divisors} of formal curves.
More specifically, a \textit{divisor of separatrices} for a foliation $\F$ at $(\C^2,p)$ is
a formal sum
\[\B = \sum_{B \in \text{Sep}_{p}(\F)} a_{B} \cdot B \]
where the coefficients $a_{B} \in \ze$ are zero except for finitely many $B \in \sep_{p}(\F)$.
We denote by $\divv_{p}(\F)$ the set of all these divisors, which turns into a group with the canonical additive structure.
We follow  the usual terminology and notation:
\begin{itemize}
\item $\B \geq 0$ denotes an \textit{effective} divisor, one whose  coefficients are all  non-negative;
\item   there is a unique decomposition $\B = \B_{0} - \B_{\infty}$, where $\B_{0}, \B_{\infty} \geq 0$ are respectively the \textit{zero}
and \textit{pole} divisors of $\B$;
\item the \textit{algebraic multiplicity} of   $\B$ is
$\nu_{p}(\B)=\displaystyle\sum_{B \in \text{Sep}_{p}(\F)} a_{B}\cdot \nu_p(B)$.
\end{itemize}
Given a    formal meromorphic equation $\hat{\Psi}$, whose irreducible components define
 separatrices  $B_i$ with multiplicities $\nu_{i}$, we associate   the divisor
$ (\hat{\Psi}) = \sum_{i} \nu_{i} \cdot B_{i}$.
A curve of separatrices $\hat{C}$, associated to a reduced equation $\hat{\Psi}$, is identified to
the divisor $\hat{C} = (\hat{\Psi})$. Such an effective divisor is named \textit{reduced}, that is,
all coefficients are either $0$ or $1$. In general,  $\B \in \divv_{p}(\F)$  is reduced if both
$\B_{0}$ and $\B_{\infty}$ are reduced effective divisors.
A divisor $\B$ is said to be \textit{adapted} to a curve of separatrices $\hat{C}$ if $\B_0 - \hat{C} \geq 0$.
Finally, the usual intersection number for formal curves $C_1=\{g_1(x,y)=0\}$ and $C_2=\{g_2(x,y)=0\}$ at $(\C^2,p)$, defined by $i_p(C_1,C_2):=\dim_{\C}\frac{\C[[x,y]]}{(g_1,g_2)}$,
is canonically extended in a bilinear way to divisors of curves.

\section{Tangency excess of a foliation}\label{tangency}
\par Let $\F$ be a germ of a singular foliation at $(\C^2,p)$ with a minimal reduction process $\pi:(\tilde{X},\D)\to (\C^2,p)$ and let $\tilde{\F} = \pi^{*} \F$ be the strict transform foliation of $\F$.
A saddle-node  singularity $q \in \sing(\Ftilde)$ 
is said to be a \textit{tangent saddle-node} if  its   weak separatrix is contained in the exceptional divisor $\D$.
\par Let $B$ be an irreducible curve invariant by $\F$ at $p$. Suppose that $\{y=0\}$ is the tangent cone of $B$, then we may choose a primitive Puiseux parametrization 
$\gamma(t)=(t^n,\phi(t))$ at $p=(0,0)$ such that $n=\nu_p(B)$. The \textit{tangency index of $\F$ along $B$ at $p$} (or \textit{weak index} in \cite[page 1114]{Fernandez}) is
$$\displaystyle\operatorname{Ind}_p^{\omega}(\F):=\ord_{t}Q(\gamma(t)).$$
\par We have the following definition given by Mattei-Salem \cite{mattei} to the non-dicritical case and by Genzmer \cite{genzmer} to arbitrary foliations:

\begin{definition}
\label{def-2ndclass}
 A foliation   is  \textit{in the second class} or is \textit{of second type} if there
are no    tangent saddle-nodes in its reduction of singularities.
\end{definition}
Given a component $D \subset \D$, we denote by  $\rho(D)$ its multiplicity (following the notation of \cite[page 1424]{mol2}),    which coincides with
the algebraic multiplicity of a curve $\gamma$ at $(\mathbb{C}^{2},p)$ whose strict transform $\pi^{*} \gamma$ meets $D$ transversally
outside a corner of $\D$. The following invariant is a measure of the existence of tangent saddle-nodes in the reduction of singularities of a foliation:

\begin{definition}\label{defisecond}
{\rm
 The \emph{tangency excess} of $\F$ is defined as $\xi_p(\F)=0$, when $p$ is a reduced singularity, and, in the non-reduced case, as the number
\[\xi_{p}(\F)=\sum_{q \in \textsl{SN}(\F) }\rho(D_{q})(\text{Ind}_{q}^{w}(\Ftilde)  -1),\]
where  $\textsl{SN}(\F)$ stands for  the set of tangent saddle-nodes  on $\D$ and, if $q \in \textsl{SN}(\F)$, we denote by  $D_q$  the component of $E$ containing its weak separatrix and by
 $\text{Ind}_{q}^{w}(\Ftilde) > 1$ its weak index.
}\end{definition}
Of course, $\xi_{p}(\F) \geq 0$ and, by definition, $\xi_{p}(\F) = 0$ if and only if $\textsl{SN}(\F) = \emptyset$, that is, if and only if  $\F$ is of second type.

We recall the following object introduced in \cite{genzmer,mol2}:
\begin{definition}{\rm
\label{def-balanced-set}
  A \emph{balanced divisor of separatrices}  for $\F$  is a divisor of the form
\[ \B \ = \
\sum_{B\in {\rm Iso}_p(\F)} B+ \sum_{B\in {\rm Dic}_{p}(\F)}\ a_{B}  \cdot B,\]
where the coefficients $a_{B} \in \mathbb{Z}$  are  non-zero except for finitely many $B \in
\dic_{p}(\F)$, and for each  dicritical  component $D \subset \D$,
 the following equality holds:
\[\sum_{B \in {\text{Sep}_{p}(D)}}a_{B} = 2- \val(D).\]
  The integer $\mathrm{Val}(D)$ represents the valence of a component $D \subset \mathcal{D}$ in the reduction of singularities; that is, it is the number of components in $\mathcal{D}$ that intersect $D$ other than $D$ itself.
}\end{definition}

A balanced divisor $\B$ is called \textit{primitive} if, for every dicritical component $D \in \D$ and
every $B \in \sep_{p}(D)$, we have $-1  \leq a_{B} \leq 1$.
Recall that a balanced divisor $\B$ is \textit{adapted} to a curve of separatrices $C$ if $\B_{0} - C \geq 0$.
A \textit{balanced equation of separatrices} is a formal meromorphic function $\hat{\Psi}$ whose associated divisor is a balanced divisor of separatrices. A balanced equation is \textit{reduced}, \textit{primitive} or \textit{adapted} to a curve $C$ if the same is true for the underlying divisor.

The tangency excess measures the extent to which a balanced divisor of separatrices computes the
 algebraic multiplicity, as expressed in the following result \cite[Proposition 2.4]{genzmer} or \cite[Proposition 3.3]{mol2}:

\begin{proposition}
\label{prop:Equa-Ba} Let $\F$ be a germ of singular foliation at $(\C^2,p)$ with $\B$ is a balanced divisor of separatrices.   Denote by $\nu_{p}(\F)$ and $\nu_{p}(\B)$ their
algebraic multiplicities. Then
\begin{equation}\label{eq_segundo}
\nu_{p}(\F)=\nu_{p}(\B)-1+\xi_{p}(\F).
\end{equation}
Moreover,
\[\nu_{p}(\F)=\nu_{p}(\B)-1\]
if, and only if, $\F$ is a second type foliation.
\end{proposition}

\section{Valuation of a foliation along a component of the exceptional divisor}\label{valuation}
In this section, we introduce our primary object of study and establish the main result of this paper.
\begin{definition}\label{defi1}
Let $\F$ be a germ of singular foliation at $(\C^2,p)$ and $\pi:(\tilde{X},\D)\to(\C^2,p)$ be a minimal process of reduction of singularities of $\F$. If $D\subset \D$ is a component, the \textit{valuation} of $\F$ along $D$, denoted by $\nu_D(\F)$, is the order of vanishing of $\pi^{*}(\omega)$ along $D$, where $\omega$ is any 1-form defining $\F$.
In the same way, if $\hat{\Psi}$ is a formal meromorphic function at $(\C^2,p)$, we define the \textit{valuation} of $\hat{\Psi}$ along $D$, denoted by $\nu_D(\hat{\Psi})$, as the order of vanishing of $\pi^{*}\hat{\Psi}:=\hat{\Psi}\circ\pi$ along $D$.
\end{definition}
\par Now, we introduce $\xi_D(\F)$ as follows: in the height 1 of the blowing-up process $\pi$, i.e., $D$ is the exceptional projective line arising from the one blowing-up at $p$, we set
\begin{equation}\label{def_x}
\xi_D(\F):=\xi_p(\F).
\end{equation}
If  $D=D_k$ is obtained at height $k\geq 2$ in the blowing-up process \[\pi=\pi_1\circ\pi_2\circ\ldots\circ\pi_k\circ\ldots: (\tilde{X},\D) \to
(\C^2,p),\] and
 $D=(\pi_{k})^{-1}(q_{k-1})$, we define
\begin{equation}\label{def-xi}
\xi_D(\F):=\xi_{q_{k-1}}(\sigma_{k-1}^{*}(\F))+\sum_{D_{q_{k-1}}\in V(q_{k-1})}\xi_{D_{q_{k-1}}}(\F).
\end{equation}
where $\sigma_{k-1}^{*}(\F)$ denotes the strict transform of the foliation $\F$ by $\sigma_{k-1}:=\pi_1\circ\ldots\circ\pi_{k-1}$, and
$V(q)$ refers to the set of irreducible components $D_q$ of $\D$ that contains $q$. In this case, $V(q)$ contains at most two components and $\mathcal{D}=\sigma^{-1}_{k-1}(p)$.
\begin{remark}\label{pull}
If $\hat{\Psi}=\frac{f}{g}$ is a germ of a formal meromorphic function at $p\in\C^2$ and $\pi$ is a blow-up at $p$ with excepcional divisor $D$, then 
the strict transform of $\hat{\Psi}$ by $\pi$ is
\[\hat{\Psi}_1:=\frac{\pi^{*}(\hat{\Psi})}{h^{\nu_p(f)-\nu_p(g)}},\]
where $\{h=0\}$ is a local equation of $D$. 
\end{remark}
\par Now, we establish the main result of this paper, which is a generalization of \cite[Th\'eor\`eme 3.1.9, item (4)]{mattei} and \cite[Lemma 3.2]{genzmer}.  
\begin{maintheorem}\label{teo1}
Let $\F$ be a germ of a singular foliation at $p\in\C^2$ having $\hat{\Psi}$ as a balanced equation of separatrices. Let $\pi:(\tilde{X},\D)\to(\C^2,p)$ be a minimal process of reduction of singularities for $\F$. Then, for every component $D\subset\D$, we have
\begin{equation}\label{eq_teo}
\nu_D(\hat{\Psi})=
\begin{cases}
\nu_D(\F)+1-\xi_D(\F) & \text{if $D$ is non-dicritical};
\medskip \\
\nu_D(\F) -\xi_D(\F) & \text{if $D$ is dicritical}.
\end{cases}
\end{equation}
\end{maintheorem}
\begin{proof}
The proof is by induction on the height $k$ of the component $D$ in the blowing-up process. If $D$ is the exceptional projective line arising from the blowing-up at $p$, then $\nu_D(\hat{\Psi})=\nu_p(\hat{\Psi})$ and $\nu_D(\F)=\nu_p(\F)+\epsilon(D)$, where 
\begin{equation}\label{eq_ed}
\epsilon(D)=\begin{cases}
 0 & \text{if $D$ is non-dicritical};
\medskip \\
1 & \text{if $D$ is dicritical}.
\end{cases}
\end{equation}
  It follows by equation (\ref{eq_segundo}) that 
 \begin{equation}\label{eq_t1}
 \nu_D(\hat{\Psi})=\nu_D(\F)+1-\epsilon(D)-\xi_p(\F).
 \end{equation} 
 Since $k=1$, we have $\xi_D(\F)=\xi_p(\F)$ by (\ref{def_x}), and the theorem at height 1 follows from (\ref{eq_t1}). 
Assume that the equality (\ref{eq_teo}) holds for height $k$ and consider the component $D$ obtained by the blowing-up $\pi_{k}$ of the point $q_{k-1}$. Here, we suppose that $\pi=\pi_1\circ\ldots\circ\pi_k\circ\ldots$.
Let $\hat{\Psi}_{k-1}$ be the strict transform of $\hat{\Psi}$ by $\sigma_{k-1}=\pi_{1}\circ\ldots\circ\pi_{k-1}$ (see Remark \ref{pull}). Analyzing the behavior of the valuations of $\F$ and $\hat{\Psi}$ along $D$ by blow-ups, we have
\begin{eqnarray}
\nu_D(\F)&=&\nu_{q_{k-1}}(\sigma_{k-1}^{*}(\F))+\sum_{D_{q_{k-1}}\in V(q_{k-1})}\nu_{D_{q_{k-1}}}(\F)+\epsilon(D),\label{c1}\\
\nu_D(\hat{\Psi})&=&\nu_{q_{k-1}}(\hat{\Psi}_{k-1})+\sum_{D_{q_{k-1}}\in V(q_{k-1})}\nu_{D_{q_{k-1}}}(\hat{\Psi}),\label{c2}
\end{eqnarray}
where $\sigma_{k-1}^{*}(\F)$ denotes the strict transform of the foliation $\F$ by $\sigma_{k-1}$, and
$V(q)$ refers to the set of irreducible components $D_q$ of $\D$ that contain $q$ and $\epsilon(D)$ as above. \\
We distinguish several cases:\\
\noindent{\textit{ Part 1. $V(q_{k-1})$ consists of one component $D_{k-1}$.} }
\begin{enumerate}
\item{\textit{$D_{k-1}$ is non-dicritical.}} Let $\hat{\Psi}_{q_{k-1}}$ be the balanced equation of separatrices for $\sigma_{k-1}^{*}(\F)$ at $q_{k-1}$. Since $q_{k-1}$ is a smooth point of $D_{k-1}$, we have that $\hat{\Psi}_{q_{k-1}}$ is the germ at $q_{k-1}$ given by the product of $\hat{\Psi}_{k-1}$ and of a germ of local equation of $D_{k-1}$, i.e.,  $\hat{\Psi}_{q_{k-1}}=\hat{\Psi}_{k-1}\cdot h_{k-1}$, where $D_{k-1}=\{h_{k-1}=0\}$.
Hence
\begin{eqnarray}
\nu_{q_{k-1}}(\hat{\Psi}_{q_{k-1}})&=&\nu_{q_{k-1}}(\hat{\Psi}_{k-1})+1.\label{c3}
\end{eqnarray}
On the other hand, applying Proposition \ref{prop:Equa-Ba} to $\sigma_{k-1}^{*}(\F)$ at $q_{k-1}$ we have
\begin{eqnarray}
\nu_{q_{k-1}}(\hat{\Psi}_{q_{k-1}})&=&\nu_{q_{k-1}}(\sigma_{k-1}^{*}(\F))+1-\xi_{q_{k-1}}(\sigma_{k-1}^{*}(\F)).\label{c4}
\end{eqnarray}
Then, the relations (\ref{c2}),  (\ref{c3}), (\ref{c4}), (\ref{c1}), and the induction hypothesis imply that
\begin{eqnarray*}
\nu_D(\hat{\Psi})&=& \nu_{q_{k-1}}(\hat{\Psi}_{k-1})+\nu_{D_{k-1}}(\hat{\Psi})\\
&=& \nu_{q_{k-1}}(\hat{\Psi}_{q_{k-1}})-1+\nu_{D_{k-1}}(\hat{\Psi})\\
&=&\nu_{q_{k-1}}(\sigma_{k-1}^{*}(\F))-\xi_{q_{k-1}}(\sigma_{k-1}^{*}(\F))+\nu_{D_{k-1}}(\hat{\Psi})\\
&=&\nu_D(\F)-\epsilon(D)\underbrace{-\nu_{D_{k-1}}(\F)+\nu_{D_{k-1}}(\hat{\Psi})}-\xi_{q_{k-1}}(\sigma_{k-1}^{*}(\F))\\
&=&\nu_D(\F)-\epsilon(D)+1-\xi_{D_{k-1}}(\F)-\xi_{q_{k-1}}(\sigma_{k-1}^{*}(\F)).
\end{eqnarray*}
 We get $\xi_D(\F)=\xi_{D_{k-1}}(\F)+\xi_{q_{k-1}}(\sigma_{k-1}^{*}(\F))$ by (\ref{def-xi}), and so that 
 $\nu_D(\hat{\Psi})=\nu_D(\F)-\epsilon(D)+1-\xi_D(\F)$ proving the result for the component $D$.

\item{\textit{$D_{k-1}$ is dicritical.}} In this case, since $D_{k-1}$ is not $\sigma_{k-1}^{*}(\F)$-invariant, and $q_{k-1}$ is a smooth point of $D_{k-1}$, we have that the balanced equation of separatrices $\hat{\Psi}_{q_{k-1}}$ for $\sigma_{k-1}^{*}(\F)$ at $q_{k-1}$ is given by $\hat{\Psi}_{q_{k-1}}=\hat{\Psi}_{k-1}$. Hence
\begin{eqnarray}
\nu_{q_{k-1}}(\hat{\Psi}_{q_{k-1}})&=&\nu_{q_{k-1}}(\hat{\Psi}_{k-1}).\label{c5}
\end{eqnarray}
Then, the relations (\ref{c2}), (\ref{c5}), (\ref{c4}), (\ref{c1}), the induction hypothesis, and (\ref{def-xi}) give the result for the component $D$. In fact
\begin{eqnarray*}
\nu_D(\hat{\Psi})&=& \nu_{q_{k-1}}(\hat{\Psi}_{q_{k-1}})+\nu_{D_{k-1}}(\hat{\Psi})\\
&=& \nu_{q_{k-1}}(\hat{\Psi}_{k-1})+\nu_{D_{k-1}}(\hat{\Psi})\\
&=&\nu_{q_{k-1}}(\sigma_{k-1}^{*}(\F))+1-\xi_{q_{k-1}}(\sigma_{k-1}^{*}(\F))+\nu_{D_{k-1}}(\hat{\Psi})\\
&=&\nu_D(\F)-\epsilon(D)+1-\xi_{q_{k-1}}(\sigma_{k-1}^{*}(\F))+\underbrace{\nu_{D_{k-1}}(\hat{\Psi})-\nu_{D_{k-1}}(\F)}\\
&=&\nu_D(\F)-\epsilon(D)+1\underbrace{- \xi_{q_{k-1}}(\sigma_{k-1}^{*}(\F))-\xi_{D_{k-1}}(\F)}\\
&=&\nu_D(\F)-\epsilon(D)+1-\xi_D(\F).
\end{eqnarray*}
\end{enumerate}
\noindent{\textit{ Part 2. $V(q_{k-1})$ consists of two components $D_{k-1}$ and $D^{1}_{k-1}$.}} 
\begin{enumerate}
\item{\textit{$D_{k-1}$ and $D^{1}_{k-1}$ are non-dicritical.}}
In this case, $q_{k-1}\in D_{k-1}\cap D^{1}_{k-1}$ is a corner point, and so that the balanced equation of separatrices $\hat{\Psi}_{q_{k-1}}$ for $\sigma_{k-1}^{*}(\F)$ at $q_{k-1}$ is the product of  $\hat{\Psi}_{k-1}$ and of a germ of local equation for $D_{k-1}\cup D^1_{k-1}$, i.e.,  $\hat{\Psi}_{q_{k-1}}=\hat{\Psi}_{k-1}\cdot h_{k-1}\cdot \phi_{k-1}$, where $D_{k-1}=\{h_{k-1}=0\}$ and $D^1_{k-1}=\{\phi_{k-1}=0\}$.
Hence, 
\begin{eqnarray}
\nu_{q_{k-1}}(\hat{\Psi}_{q_{k-1}})&=&\nu_{q_{k-1}}(\hat{\Psi}_{k-1})+2.\label{c7}
\end{eqnarray}
Then, the relations (\ref{c2}), (\ref{c7}), (\ref{c4}), (\ref{c1}) and the induction hypothesis imply that
\begin{eqnarray}
\nu_D(\hat{\Psi})&=& \nu_{q_{k-1}}(\hat{\Psi}_{k-1})+\nu_{D_{k-1}}(\hat{\Psi})+\nu_{D^{1}_{k-1}}(\hat{\Psi})\nonumber\\
&=& \nu_{q_{k-1}}(\hat{\Psi}_{q_{k-1}})-2+\nu_{D_{k-1}}(\hat{\Psi})+\nu_{D^{1}_{k-1}}(\hat{\Psi})\nonumber\\
&=&\nu_{q_{k-1}}(\sigma_{k-1}^{*}(\F))+1-\xi_{q_{k-1}}(\sigma_{k-1}^{*}(\F))-2+\nu_{D_{k-1}}(\hat{\Psi})+\nu_{D^{1}_{k-1}}(\hat{\Psi})\nonumber\\
&=&\nu_D(\F)-\epsilon(D)-1-\xi_{q_{k-1}}(\sigma_{k-1}^{*}(\F))\nonumber\\
& & +\nu_{D_{k-1}}(\hat{\Psi})-\nu_{D_{k-1}}(\F)+\nu_{D^{1}_{k-1}}(\hat{\Psi})-\nu_{D^{1}_{k-1}}(\F)\nonumber\\
&=&\nu_D(\F)-\epsilon(D)-1-\xi_{q_{k-1}}(\sigma_{k-1}^{*}(\F))\nonumber\\
& & +\left(1-\xi_{D_{k-1}}(\F)\right)+\left(1-\xi_{D_{k-1}^{1}}(\F)\right)\nonumber\\
&=&\nu_D(\F)-\epsilon(D)+1\nonumber\\
& & -\left(\xi_{q_{k-1}}(\sigma_{k-1}^{*}(\F))+\xi_{D_{k-1}}(\F)+\xi_{D_{k-1}^{1}}(\F)\right)\label{eq_45}.
\end{eqnarray}
 Thus, the result follows from (\ref{eq_45}) and (\ref{def-xi}):
\begin{eqnarray*}
\nu_D(\hat{\Psi})
&=& \nu_D(\F)-\epsilon(D)+1-\xi_D(\F).
\end{eqnarray*}
\item{\textit{$D_{k-1}$ is dicritical and $D^{1}_{k-1}$ is non-dicritical.}}
This case can be handled similarly.
\item{\textit{$D_{k-1}$ and $D^{1}_{k-1}$ are dicritical.}} 
In this case, $q_{k-1}\in D_{k-1}\cap D^{1}_{k-1}$ is a corner point, and 
since $D_{k-1}$ and $D^{1}_{k-1}$ are dicritical, we have the balanced equation of separatrices $\hat{\Psi}_{q_{k-1}}$ for $\sigma_{k-1}^{*}(\F)$ at $q_{k-1}$ is  $\hat{\Psi}_{k-1}$.
Hence, 
\begin{eqnarray}
\nu_{q_{k-1}}(\hat{\Psi}_{q_{k-1}})&=&\nu_{q_{k-1}}(\hat{\Psi}_{k-1}).\label{c13}
\end{eqnarray}
Then, the relations (\ref{c2}), (\ref{c13}), (\ref{c4}), (\ref{c1}) and the induction hypothesis imply that
\begin{eqnarray}
\nu_D(\hat{\Psi})&=& \nu_{q_{k-1}}(\hat{\Psi}_{k-1})+\nu_{D_{k-1}}(\hat{\Psi})+\nu_{D^{1}_{k-1}}(\hat{\Psi})\nonumber\\
&=& \nu_{q_{k-1}}(\hat{\Psi}_{q_{k-1}})+\nu_{D_{k-1}}(\hat{\Psi})+\nu_{D^{1}_{k-1}}(\hat{\Psi})\nonumber\\
&=&\nu_{q_{k-1}}(\sigma_{k-1}^{*}(\F))+1-\xi_{q_{k-1}}(\sigma_{k-1}^{*}(\F))+\nu_{D_{k-1}}(\hat{\Psi})+\nu_{D^{1}_{k-1}}(\hat{\Psi})\nonumber\\
&=&\nu_D(\F)-\epsilon(D)+1-\xi_{q_{k-1}}(\sigma_{k-1}^{*}(\F))\nonumber\\
& & +\nu_{D_{k-1}}(\hat{\Psi})-\nu_{D_{k-1}}(\F)+\nu_{D^{1}_{k-1}}(\hat{\Psi})-\nu_{D^{1}_{k-1}}(\F)\nonumber\\
&=&\nu_D(\F)-\epsilon(D)+1-\xi_{q_{k-1}}(\sigma_{k-1}^{*}(\F))-\xi_{D_{k-1}}(\F)-\xi_{D_{k-1}^{1}}(\F)\nonumber\\
&=&\nu_D(\F)-\epsilon(D)+1\nonumber\\
& & -\left(\xi_{q_{k-1}}(\sigma_{k-1}^{*}(\F))+\xi_{D_{k-1}}(\F)+\xi_{D_{k-1}^{1}}(\F)\right)\label{eq_45}.
\end{eqnarray}
 Thus, the result follows from (\ref{eq_45}) and (\ref{def-xi}):
\begin{eqnarray*}
\nu_D(\hat{\Psi})
&=& \nu_D(\F)-\epsilon(D)+1-\xi_D(\F).
\end{eqnarray*}

\end{enumerate}
\end{proof}
\par In order to illustrate Theorem \ref{teo1} we consider the family of dicritical foliations given in \cite[Example 6.5]{FP-GB-SM2021}. 
\begin{example}
Let $\lambda\in\C$ and $k\geq 3$ integer. Let $\F_k$ be the singular foliation at $(\C^2,0)$ defined by
\[\omega_k=y(2x^{2k-2}+2(\lambda+1)x^2y^{k-2}-y^{k-1})dx+x(y^{k-1}-(\lambda+1)x^2y^{k-2}-x^{2k-2})dy.\]
The foliation $\F_k$ is  dicritical and is not of second type. After one blow-up $\pi_1$, appears a dicritical component $D_1=\pi_1^{-1}(0)$ and the strict transform of $\F_k$ by $\pi_1$ has a unique non-reduced singularity $q\in D_1$.  A further blow-up $\pi_2$ applied to $q$ produces a non-dicritical component $D_2=\pi_2^{-1}(q)$, and the reduction of singularities of $\F_k$ is achieved. 
A balanced equation of separatrices for $\F_k$ is $\hat{\Psi}(x,y)=xy$.
Moreover, a straightforward calculation reveals that
\[\nu_{D_1}(\F_k)=k+1,\,\,\, \xi_0(\F_k)=k-1,\,\,\, \nu_{D_{1}}(\hat{\Psi})=2,\]
\[\nu_{D_2}(\F_k)=2k,\,\,\, \xi_q(\F_k)=k-1,\,\,\, \nu_{D_{2}}(\hat{\Psi})=3.\]
Therefore, $\xi_{D_1}(\F_k)=k-1$ and $\xi_{D_2}(\F_k)=2k-2$. Since
 \[2=\nu_{D_{1}}(\hat{\Psi})=\nu_{D_1}(\F_k)-\xi_{D_1}(\F_k)=k+1-(k-1)\]
 and 
 \[3=\nu_{D_{2}}(\hat{\Psi})=\nu_{D_2}(\F_k)+1-\xi_{D_2}(\F_k)=2k+1-(2k-2).\] 
 We verify that Theorem \ref{teo1} holds.
\end{example}
\par As a consequence of \cite[Theorem 3]{CLS}, we have $\nu_p(\F)\geq\nu_p(\Gamma)-1$ if the foliation $\F$ is non-dicritical at $p\in\C^2$, and $\Gamma$ is the union of all separatrices through $p$. Then, since $\xi_D(\F)\geq 0$, we can obtain a similar inequality for $\nu_D(\F)$ and $\nu_D(\hat{\Psi})$.
\begin{corollary}
Let $\F$ be a germ of a singular foliation at $p\in\C^2$ having $\hat{\Psi}$ as a balanced equation of separatrices. Let $\pi:(\tilde{X},\D)\to(\C^2,p)$ be a minimal process of reduction of singularities for $\F$. Then for every component $D\subset\D$, we have
\[\nu_D(\hat{\Psi})-1+\epsilon(D)\leq\nu_D(\F),\]
where $\epsilon(D)$ is as in (\ref{eq_ed}).
\end{corollary}
\section{A remark on projective foliations}
\label{projective}
\par In this section, under certain assumptions, we derive inequalities involving the valuation, tangency excess, and degree of a holomorphic foliation $\F$ on the complex projective plane $\mathbb{P}^{2}_\C$. Given that these invariants are related to the algebraic multiplicity $\nu_p(\F)$, we believe they hold independent interest; see, for instance, \cite{Jose} and \cite[Section 3]{oziel}.
\par A holomorphic foliation $\F$ on $\mathbb{P}^{2}_{\C}$ of degree $d\geq 0$ is a foliation defined by a polynomial 1-form $\Omega=A(x,y,z)dx+B(x,y,z)dy+C(x,y,z)dz,$ where 
$A,B,C$ are complex homogeneous polynomials of degree $d+1$, satisfying two conditions:
\begin{enumerate}
    \item the integrability condition $\Omega\wedge d\Omega =0$,
    \item the Euler condition $Ax+By+Cz=0$.
\end{enumerate}
The singular set of $\F$ is by definition 
\[\sing(\F)=\{p\in\mathbb{P}^2_{\C}:A(p)=B(p)=C(p)=0\}.\]

According to Mendes-Sad \cite[page 94]{mendes}, an analytic curve $S$ contained in a complex surface $M$ is said to be a \textit{special invariant curve} for a holomorphic foliation $\G$ if $S$ is smooth, isomorphic to $\mathbb{P}^1_{\C}$, $\G$-invariant, $\sing(\G)\cap S=\{q\}$ and there are local coordinates for which $q=(0,0)$ and $f(x,y)=x^m\cdot y$ $(m\in\mathbb{N})$ is a local holomorphic first integral for $\G$.
\par Now, we consider a holomorphic foliation $\F$ of $\mathbb{P}^2_{\C}$ of degree $d\geq 1$, and a minimal process of reduction of singularities $\pi:(\tilde{X},\D)\to\mathbb{P}^2_{\C}$ for $\F$ such that the strict transform foliation of $\F$ by $\pi$ is free   of special invariant curves. Then, it follows from Mendes-Sad's Theorem \cite[Theorem 1]{mendes} that 
\begin{equation}\label{eq_sad}
\sum_{p\in\sing(\F)}\sum_{D\subset\D}(\nu_D(\F_p)-1)^2\leq (d-1)^2,
\end{equation}
where $\F_p$ is the germ of $\F$ at $p\in\sing(\F)$. 

We have the following corollary.
\begin{corollary}
Let $\F$ be a holomorphic foliation of $\mathbb{P}^2_{\C}$ of degree $d\geq 1$, and let $\pi:(\tilde{X},\D)\to\mathbb{P}^2_{\C}$ be a minimal process of reduction of singularities for $\F$. For each $p\in\sing(\F)$, let $\hat{\Psi}_p$ be a balanced equation of separatrices for the germ $\F_p$ of $\F$ at $p$.  Suppose that the strict transform foliation of $\F$ by $\pi$ is free   of special invariant curves. Then
 \begin{eqnarray}\label{eq_final}
     & &\sum_{p\in\sing(\F)}
    \sum_{D\subset\D\,\,\text{non-dicritical}}(\nu_D(\hat{\Psi}_p)+\xi_D(\F_p)-2)^2  \nonumber
     \\ 
     & & +
    \sum_{p\in\sing(\F)}
    \sum_{D\subset\D\,\,\text{dicritical}}(\nu_D(\hat{\Psi}_p)+\xi_D(\F_p)-1)^2 \leq (d-1)^2
\end{eqnarray}
\end{corollary}
\begin{proof}
Dividing the sum (\ref{eq_sad}) in two parts (non-dicritical and dicritical component of $\D$) we get
\begin{eqnarray*}
   & &\sum_{p\in\sing(\F)}\sum_{D\subset\D\,\,\text{non-dicritical}}(\nu_D(\F_p)-1)^2\\
    & & +\sum_{p\in\sing(\F)}
    \sum_{D\subset\D\,\,\text{dicritical}}(\nu_D(\F_p)-1)^2\leq (d-1)^2.
    \end{eqnarray*}
    The inequality (\ref{eq_final}) follows applying Theorem \ref{teo1} to $\nu_D(\F_p)$.
\end{proof}

\subsection*{Acknowledgements}
The first-named author acknowledges support from CNPq Projeto Universal 408687/2023-1 ``Geometria das Equa\c{c}\~oes Diferenciais Alg\'ebricas" and CNPq-Brazil PQ-306011/2023-9. This work was funded by the Direcci\'on de Fomento de la Investigaci\'on at the PUCP through grant DFI-2023-PI0979.  Finally, the authors thank the anonymous referee, whose remarks allowed them to improve the presentation.


\end{document}